\documentclass[12pt,reqno]{amsart}
\usepackage{bm} 
\usepackage{amsfonts}
\usepackage{mathrsfs}
\usepackage{yfonts}
\usepackage{url}

\usepackage{mathtools}

\usepackage[left=2.5cm, top=3cm, bottom=3cm, right=2.5cm]{geometry}
\usepackage{amssymb}  		% Activate to begin paragraphs with an empty line rather than an indent
\usepackage{graphicx}
\usepackage{verbatim}
\usepackage{enumerate}	
\usepackage[english]{babel}
\usepackage{amsmath}
\usepackage{bussproofs}	
\usepackage{quoting}
\usepackage{array,booktabs}

\usepackage[shortlabels]{enumitem}
\usepackage{units}
\usepackage{xspace}\xspaceaddexceptions{\,}
\usepackage{multicol}
\usepackage{booktabs,caption}
   \usepackage{tikz-cd}
\quotingsetup{font=small}	% Use pdf, png, jpg, or eps§ with pdflatex; use eps in DVI mode
\usepackage{amsthm}								% TeX will automatically convert eps --> pdf in pdflatex	
\usepackage[pdftex,bookmarks,bookmarksnumbered,linktocpage,   %%  customize to your liking
         colorlinks,linkcolor=blue,citecolor=blue]{hyperref}

\theoremstyle{definition}
\newtheorem{definition}{Definition}
\newtheorem{example}[definition]{Example}

\newtheorem{remark}[definition]{Remark}

\theoremstyle{plain}
\newtheorem{lemma}[definition]{Lemma}
\newtheorem{proposition}[definition]{Proposition}
\newtheorem{theorem}[definition]{Theorem}
\newtheorem{corollary}[definition]{Corollary}

\newcommand\A{{\mathbf A}}

\newcommand\R{{\mathbf R}}

\newcommand\vx{\vec{x}}

\newcommand\Int{{\mathrm {Int}}}

\newcommand{\M}{M(x_{0},\dots,x_{n})}
\newcommand{\Ge}{\langle x_{i}, x_{j}, x_{0}\rangle}

\makeatletter
\@namedef{subjclassname@2020}{%
  \textup{2020} Mathematics Subject Classification}
\makeatother

\title[]{Embeddings of metric Boolean algebras in $\mathbb{R}^{N}$}

\author{Stefano Bonzio}
\address{Stefano Bonzio, Department of Mathematics and Computer Science \\
         University of Cagliari, Italy.}
\email{stefano.bonzio@unica.it}
\author{Andrea Loi}
\address{Andrea Loi, Department of Mathematics and Computer Science \\
         University of Cagliari, Italy.}
         \email{loi@unica.it}

\date{}

\begin{document}

\maketitle

\begin{abstract}
A Boolean algebra $\A$ equipped with a (finitely-additive) positive probability measure $m$ can be turned into a metric space $(\A , d_{m})$, where $d_{m}(a,b)= m ((a\wedge\neg b)\vee(\neg a\wedge b))$, for any $a,b\in A$, sometimes referred to as \emph{metric Boolean algebra}. In this paper, we study under which conditions the space of atoms of a finite metric Boolean algebra can be isometrically embedded in $\mathbb{R}^{N}$ (for a certain $N$) equipped with the Euclidean metric. In particular, we characterize the topology of the positive measures over a finite algebra $\A$ such that the metric space $(\mathsf{At}(\A), d_m)$ embeds isometrically in $\mathbb{R}^{N}$.   
\end{abstract}

\section{Introduction}

Every Boolean algebra $\A$ equipped with a positive probability measure $m$ is a metric space, where the Kolmogorov distance between two elements $a,b\in A$ is given by the value (in $[0,1]$) assigned by $m$ to the symmetric difference between $a$ and $b$, namely the element $(a\wedge\neg b)\vee (\neg a\wedge b)$. Algebras ``metrized'' by positive measures are called \emph{metric Boolean algebras}, a nomenclature introduced by Kolmogorov \cite{Kolmogorov}, or \emph{normed Boolean algebras} \cite{BooleanAlgebrainAnalysis}. 

The algebraic and topological interest around metric (or normed) Boolean algebras mainly arises from probability theory and its subfield of stochastic geometry \cite{normedBA}. Moreover, these structures have recently found interesting applications in the theory of random sets (see e.g. \cite{RandomSets}), which can be represented as a random element taking values in a normed Boolean algebra \cite{gotovac2017new}. Applications of the theory of metric Boolean algebras mostly rely on the idea that they are structures well suited to treat the probability of (two-valued) events (see e.g. \cite{BooleanAlgebrainAnalysis}). However, the observation that positive probability measures turn an algebra into a metric space adds an importance going far beyond logic and probability. Indeed, in the last few years, the interest around the geometry of discrete metric spaces has relevantly increased. Due to innovative ideas of Gromov \cite{Gromovbook} and others \cite{BACAK,Burago, JOHARINAD}, the study of traditionally relevant geometrical properties usually conceived for Riemannian manifolds, such as length of curves, measures, sectional and Ricci curvatures, has been fruitfully extended to discrete spaces. In this context, the present work focuses on expanding the geometry of metric Boolean algebras and, in particular, to understand for which choices of measures $m$ these metric spaces resemble Euclidean spaces, in the sense that the subspace of atoms of a metric Boolean algebras can be isometrically embedded in the prototypical example of metric space, $\mathbb{R}^{N}$ (for some $N$) equipped with the Euclidean metric. The choice of the space of the atoms (forcing to confine our attention to atomic algebras only) is driven, on the one side, by its relevance in probability theory, as, in an atomic $\sigma$-algebra of events, the probability of any event depends on the probability distribution on the set of atoms and, on the other, by the impossibility of getting an isometric embedding of the entire algebra (see Remark \ref{rem: no immersione in R_n}).

The paper is structured as follows: in Section \ref{sec: preliminari} are introduced all the preliminary notions necessary to go through the reading of the whole paper. Section \ref{sec: immersioni} discusses the details of isometric embeddings of generic metric spaces into Euclidean spaces (ruled by Morgan's theorem) and shows the existence of a probability measure allowing the embedding of the atoms of a metric Boolean algebra in $\mathbb{R}^{N}$, for an appropriate $N$ (Corollary \ref{cor: immersione atomi con stessa probabilita}). Section \ref{sez: principale} contains the main findings of the paper, namely the study of the topology of the space of the measures  for which the metric space of the atoms of a finite Boolean algebra can be isometrically embedded in $\mathbb{R}^{N}$, for some $N$: the main finding is that this space is contractible, while its complement is simply connected but not contractible. Finally, Section \ref{sec: Appendice} contains the proof of the very useful technical Lemma \ref{lemma: determinante}, which is applied throughout the whole paper, for establishing the existence/non-existence of isometric embeddings in $\mathbb{R}^{N}$.

\section{Preliminaries: metric Boolean algebras}\label{sec: preliminari}

Let $\A$ be a Boolean algebra equipped with a \emph{strictly positive} (finitely additive) probability measure, i.e. a map $m\colon\A\to [0,1]$ such that: 
\begin{enumerate}
\item $m(\top) = 1$, 
\item $m(a\vee b) = m(a) + m(b)$, for every $a,b\in A$ such that $a\wedge b = \perp$, 
\item $m(a) > \perp$, for every $a\in A$, $a\neq \perp$. 
\end{enumerate}
We indicate with $\top$ and $\perp$ the top and bottom element, respectively, of a Boolean algebra (to avoid confusion with the numbers $0,1$).
\vspace{10pt}

The following recalls the well-known properties of probability measures. 
\begin{proposition}\label{prop: proprieta m}
Let $m$ be a finitely additive probability measure over a Boolean algebra $\A$, then the following hold, for every $a,b\in A$: 
\begin{enumerate}
\item $m$ is monotone, i.e. if $a\leq b$ then $m(a)\leq m(b)$;
\item $m(a) + m(b) = m(a\vee b) + m(a\wedge b)$; 
\item $m(a\vee b)\leq m(a) + m(b)$;
\item $m(a') = 1-m(a)$;
\item if $m$ is strictly positive, $a < b$ implies $m(a) < m(b)$; 
\item If $\A$ is finite then $\displaystyle\sum_{i={1}}^{k} m(a_{i}) = 1$, where $\{a_{1},\dots,a_{k}\}$ is the set of atoms of $\A$.
\end{enumerate}
\end{proposition}

\begin{remark}\label{rem: spazio metrico}
Let $\A$ be a Boolean algebra equipped with a strictly positive (finitely additive) probability measure $m$. Then ($\A, d_{m}$) is a metric space, where 
\[
d_{m}(a,b)\coloneqq m(a\vartriangle b) = m((a\wedge b')\vee(a'\wedge b)),
\] 
for every $a,b\in A$.\\
\noindent
1) $d_{m}(a,b) = d_{m}(b,a)$ is immediate by commutativity of $\wedge$ and $\vee$. \\
\noindent
2) for triangular inequality, first recall that, in any Boolean algebra $\A$, it holds $a\wedge c'\leq (a\wedge b')\vee (b\wedge c')$, for any $a,b,c\in A$. We have to show that $m(a\vartriangle c)\leq m(a\vartriangle b) + m(b\vartriangle c)$. Observe that $a\vartriangle c = (a\wedge c')\vee (a'\wedge c) \leq (a\wedge b')\vee (b\wedge c')\vee (a'\wedge b)\vee(b'\wedge c) = ((a\wedge b')\vee(a'\wedge b))\vee((b\wedge c')\vee(b'\wedge c)) = (a\vartriangle b)\vee(b\vartriangle c)$. It then follows, from Proposition \ref{prop: proprieta m}, that $m(a\vartriangle c)\leq m((a\vartriangle b)\vee(b\vartriangle c))\leq m(a\vartriangle b) + m(b\vartriangle c)$.\\
\noindent
3) $d_{m}(a,b) = 0$ if and only if $a=b$. Right to left is obvious. For the other direction, suppose $m(a\vartriangle b)= 0$, then, since $m$ is strictly positive, $(a\wedge b')\vee(a'\wedge b) = \perp$, i.e. $a\wedge b' = \perp$ and $b\wedge a'= \perp$, and this means that $b'= a'$, thus $a = b$. 
\qed
\end{remark}

From now on, we always intend a Boolean algebra $\A$ equipped with a strictly positive (finitely additive) probability measure $m$ and refers to it as \emph{metric} Boolean algebra, namely as the metric space $(\A,d_{m})$ (we will not make explicit reference to the metric $d_{m}$). Observe that the assumption that $m$ is strictly positive is crucial to have a metric space: in case $m$ is \emph{not} strictly positive then $(\A, d_{m})$ is a pseudo-metric space.

In general, not every Boolean algebra can be equipped with a strictly positive (finitely additive) probability measure: a characterization of those Boolean algebras for which such measures exist is due to Kelley \cite[Theorem 4]{Kelley59} (see also \cite{Jech}). For the purpose of the present paper, we observe that every atomic Boolean algebra whose sets of atoms is numerable admits a strictly positive measure \cite[Theorem 2.5]{HornTarski}.

The basic properties of the term-operation $\vartriangle$ are recalled in the following. 
\begin{proposition}\label{prop: proprieta triangolo}
For any $a,b\in A$, the following hold: 
\begin{enumerate}
\item $a\vartriangle \perp = a$, 
\item $a\vartriangle \top = a'$,
\item $a'\vartriangle b' = a\vartriangle b$, 
\item $a\vartriangle a' = \top$. 
\end{enumerate}
\end{proposition}

We now prove an easy fact about atoms of (atomic) Boolean algebras that will be used in the next section. We indicate by $\mathsf{At}(\A)$ the set of atoms of an atomic Boolean algebra $\A$.

\begin{lemma}\label{lemma: distanza tra gli atomi}
Let $\A$ be an atomic metric Boolean algebra and let $a,b\in\mathsf{At}(\A)$ (with $a\neq b$). Then $d_{m}(a,b) = m(a) + m(b)$.
\end{lemma}
\begin{proof}
Let $a,b\in\mathsf{At}(\A)$, with $a\neq b$. % Two cases may arise: either $a\leq b'$, or $a\wedge b' = \perp$. However, the latter can be excluded. Indeed, suppose by contradiction that $a\wedge b' = \perp$; then $a'= a'\vee \perp = a'\vee(a\wedge b') = (a'\vee a)\wedge (a'\vee b') = \top\wedge (a'\vee b') = a'\vee b'$, i.e. $b'\leq a'$ whence $a\leq b$, in contradiction with the fact that $a$ and $b$ are distinct atoms.\\
Since $a\wedge b = \perp$, then $a\leq b'$ and $b\leq a'$; thus $d_{m}(a,b) = m((a\land b')\vee(a'\land b)) = m(a\land b') + m(a'\land b) = m(a) + m(b)$.
\end{proof}

\section{Isometric embeddings in $\mathbb{R}^{N}$}\label{sec: immersioni}

It is natural to wonder whether the metric space $(\A,d_{m})$ embeds isometrically into $\mathbb{R}^{N}$ (for some $N$) with the Euclidean metric. Unfortunately this is never the case. 

\begin{remark}\label{rem: no immersione in R_n}
A Boolean algebra $\A$ (with $|A| > 2$) can \emph{not} be isometrically embedded in $\mathbb{R}^{N}$, for any $N\in\mathbb{N}$, equipped with the Euclidean distance. Indeed, let $a\in A$ be any element different from the constants. Then $d_{m}(a,0) = m (a) $, $d_{m}(a,1) = m(a') = 1 - m(a)$ and $d_{m}(0,1) = 1$, which implies that any isometric embedding maps $a,0,1$ on the same line, where (for analogous reasoning) should lie also $a'$. Since $d_{m}(a,a') = 1$, the only possibility is setting $\iota\colon\A\to\mathbb{R}^{N}$, $\iota(a) = \iota(0)$ and $\iota(a')=\iota(1)$ (or, viceversa), but such an embedding can not be isometric, as $|\iota(a)-\iota(0)| = 0$ while $d_{m}(a,0)\neq 0$.
\qed
\end{remark}

We now turn our attention to a relevant subspace of a metric atomic Boolean algebra, namely the space of $\mathsf{At}(\A)$ of its atoms (with the metric $d_{m}$). \\
\vspace{5pt}

\textbf{Question:} Is there a (finitely additive) strictly positive probability measure $m$ such that the space $(\mathsf{At}(\A), d_m )$ can be isometrically embedded in $\mathbb{R}^{N}$, for some $N$? \\

\noindent
Embeddings of generic metric spaces into $\mathbb{R}^{N}$ are ruled by a theorem of Morgan \cite{Morgan}. We recall some notions relevant to introduce it. % a result stating the conditions for a metric space to be embedded in $\mathbb{R}^{n}$.

\begin{definition}\label{def: flat}
A metric space $(X,d)$ is flat if the determinant of the $n\times n$ matrix $\M$, whose generic entry is $ \langle x_{i}, x_{j}, x_0\rangle= \frac{1}{2}(d(x_{0},x_{i})^{2} + d(x_{0},x_{j})^{2} - d(x_{i},x_{j})^{2}) $, is non-negative for every $n$\emph{-simplex}, namely every choice of $n+1$ points $\{x_0 , \dots , x_{n}\}$ in $X$.
\end{definition}

\begin{definition}\label{def: dimensione}
The dimension of a space $(X,d)$ is the greatest $N $ (if exists) such that there exists a $N$-simplex with positive determinant.
\end{definition}
\begin{theorem}[Morgan \cite{Morgan}]\label{th: teorema di Morgan}
A metric space $(X,d)$ embeds in $\mathbb{R}^{N}$ if and only if it is flat and has dimension less or equal to $N$.
\end{theorem}

Morgan's theorem is constructive. Indeed, given $(X,d)$ flat and with dimension $N$, then the embedding into $\mathbb{R}^{N}$ is given by:
$$ f\colon X\to \mathbb{R}^{N} $$
$$ x\mapsto (\langle x, x_{1}, x_0\rangle,\dots, \langle x, x_{N}, x_0\rangle), $$

for a generic $n$-simplex $\{x_{0}, x_{1}, \dots, x_{N}\} $.\\

\noindent
From now on, our analysis will be confined to finte metric Boolean algebras.

\begin{remark}\label{rem: nostra matrice}
In order to simplify notation, given a (finite) metric Boolean algebra $\A$ with $k+1$ atoms $\mathsf{At}(\A) = \{a_0, a_{1},\dots, a_{k}\}$, we set $x_\alpha = m(a_{\alpha})$ (thus $x_{\alpha} > 0$, for every $ \alpha\in\{0,1,\dots, k\}$).
The matrix $\M$, $2\leq n\leq k$, introduced in Definition \ref{def: flat}, associated to $\A$ has generic entry
\begin{equation}\label{eq: generica entrata matrice}
\Ge = (x_0 + x_i)^{2}\delta_{ij} + (x_{0}^{2} + x_{0}x_{i} + x_{0}x_{j} - x_{i}x_{j})(1-\delta_{ij}).
\end{equation}
\noindent
Indeed, $\Ge = \frac{1}{2}(d_{m}(a_{0},a_{i})^{2} + d_{m}(a_{0},a_{j})^{2} - d_{m}(a_{i},a_{j})^{2}) $, and, for $i = j$, $d_{m}(a_i , a_j ) = 0$, hence, by Lemma \ref{lemma: distanza tra gli atomi}, $\Ge = \frac{1}{2}(2(x_{0} + x_{i})^{2}) = (x_{0} + x_{i})^{2}$. Else, for $i\neq j$, $\Ge = \frac{1}{2}((x_{0} + x_{i})^{2} + (x_{0} +x_{j})^{2} - (x_{i} +x_{j})^{2}) = \frac{1}{2}(x_{0}^{2} + x_{i}^{2} + 2x_{0}x_{i} + x_{0}^{2} + x_{j}^{2} + 2x_{0}x_{j} - x_{i}^{2} - x_{j}^{2} - 2x_{i}x_{j}) =  x_{0}^{2} + x_{0}x_{i} + x_{0}x_{j} - x_{i}x_{j}$.
\qed
%Indeed, by Lemma \ref{lemma: distanza tra gli atomi}, $M_{ij} = \frac{1}{2}(d(a_{0},a_{i})^{2} + d(a_{0},a_{j})^{2} - d(a_{i},a_{j})^{2}) = \frac{1}{2}((x_{0} + x_{i})^{2} + (x_{0} +x_{j})^{2} - (x_{i} +x_{j})^{2})$, thus, for $i=j$, $M_{ij} = \frac{1}{2}(2(x_{0} + x_{i})^{2}) = (x_{0} + x_{i})^{2}$ and, for $i\neq j$, $M_{ij} = \frac{1}{2}((x_{0} + x_{i})^{2} + (x_{0} +x_{j})^{2} - (x_{i} +x_{j})^{2}) = \frac{1}{2}(x_{0}^{2} + x_{i}^{2} + 2x_{0}x_{i} + x_{0}^{2} + x_{j}^{2} + 2x_{0}x_{j} - x_{i}^{2} - x_{j}^{2} - 2x_{i}x_{j}) =  x_{0}^{2} + x_{0}x_{i} + x_{0}x_{j} - x_{i}x_{j}$.
\end{remark}

\begin{lemma}\label{lemma: determinante}
Let $\A$ be a finite metric atomic Boolean algebra with $k+1$ atoms and $\M$, $2\leq n\leq k$ be the matrix defined above. Then % \textcolor{red}{METTEREI L'INDICE $\alpha =0, \dots n$ e $i, j=1, \dots n$.}
\begin{equation}\label{fondlemma}
\det (\M)=2^{n-1}\left[\left(\sum_{\alpha=0}^n x_0\cdot\cdots \cdot\hat x_\alpha\cdot \cdots \cdot x_n\right)^2-(n-1)\left(\sum_{\alpha=0}^n x_0^2\cdot\cdots\cdot\hat x_\alpha^2\cdots\cdot x_n^2\right)\right],
\end{equation}
where $\hat x_i$ means that $x_i$ has to be omitted.
\end{lemma}
\begin{proof}
See Appendix \ref{sec: Appendice}.
\end{proof}

\begin{remark}\label{rem: xi non per forza nascono da probabilita}
In the proof of Lemma \ref{lemma: determinante}, it is actually enough to assume that $x_{0},\dots,x_{n}$ satisfies the property introduced in \eqref{eq: generica entrata matrice}; thus from now on, we will weaken the assumption that $x_{0},\dots,x_{n}$ are probabilities: we will just assume that they are elements in $(0,1)$ satisfying equation \eqref{eq: generica entrata matrice}. 

%\textcolor{blue}{Osserviamo che, nella dimostrazione del Lemma \ref{lemma: determinante} abbiamo usato il fatto che $x_{\alpha}$ (per $\alpha\in\{...\}$) sono assegnazioni di probabilit\`a (sugli atomi) SOLO nel senso che si sommano (vedi Remark \ref{rem: nostra matrice}). Da qui in poi indeboliremo quindi l'assunzione e $x_{\alpha}$ sar\`a semplicemente un elemento $(0,1)$ con la ``propriet\`a della somma'' (Lemma \ref{lemma: distanza tra gli atomi}). La ragione di tale scelta sar\`a poi chiara nella prossima sezione.}
\end{remark}

Lemma \ref{lemma: determinante} allows to provide a positive answer to the above state question. Indeed, as a corollary we get the embedding of the atoms from a finite metric Boolean algebra which are assigned with the same probability, according to the \emph{principle of indifference}.
\begin{corollary}\label{cor: immersione atomi con stessa probabilita}
Let $\A$ be a finite metric Boolean algebra with $k+1$ atoms ($|A| = 2^{k+1}$) and $m$ a finitely additive probability measure such that $m(a_{i}) = \frac{1}{k+1}$, for every $a_{i}\in\mathrm{At}(\A)$. Then $\mathrm{At}(\A)$ embeds in $\mathbb{R}^{k}$ with the Euclidean metric.
\end{corollary}
\begin{proof}
In virtue of Theorem \ref{th: teorema di Morgan}, it is enough to show that $\mathrm{At}(\A)$ is flat and has dimension $k$. By assumption $x_{0} = x_1 = \dots = x_k = \frac{1}{n+1} > 0$. Hence, applying Lemma \ref{lemma: determinante}, we get 
\begin{align*}
\det(\M) =& 2^{n-1}\left[ \left((n+1)x_{0}^{n}\right)^{2} - (n-1)(n+1)x_{0}^{2n}\right]  
\\  =&   2^{n-1}\left[ (n+1)^{2}x_{0}^{2n} - (n^{2}- 1)x_{0}^{2n}\right]
\\ = & 2^{n-1}2nx_{0}^{2n} > 0, 
\end{align*}
for every $2\leq n\leq k$.
\end{proof}

It follows from Theorem \ref{th: teorema di Morgan} and Lemma \ref{lemma: determinante}  that the space $\mathrm{At}(\A)$ of a Boolean algebra $\A$ with 3 atoms ($|A|=8$) embeds in $\mathbb{R}^{2}$ (it always true that a metric space of cardinality $3$ embeds isometrically in $\mathbb{R}^{2}$!). %\textcolor{blue}{For higher dimensions, this might not be the case. }

\begin{remark}\label{m2}
It is easy to check that $\det (M(x_{0},x_{1}, x_{2}))>0$, a property that we will apply insofar with no explicit mention.
\end{remark}

\begin{remark}\label{rem:omogeneo}
Observe that for any $\lambda\in\mathbb{R}$, $\det (M(\lambda x_0,\dots,\lambda x_{n})) = \lambda^{2n}\det( \M)$. Moreover, $\forall n$ ($2\leq n\leq k$), $\M\geq 0$ if and only if $\displaystyle (\sum_{\alpha=0}^{n} \frac{1}{x_\alpha})^2 - (n-1)\sum_{\alpha=0}^{n}\frac{1}{x_\alpha^{2}}\geq 0$.
\end{remark}

It is not always the case that the space of atoms of a (finite) metric Boolean algebra embeds in $\mathbb{R}^{N}$. In the following we consider the probability assignment in accordance with the \emph{binomial distribution}.

\begin{example}\label{ex: misura che non si immerge}
The binomial distribution (with parameters $n$ and $p$) is the probability distribution of the number of successes in a sequence of $n$ independent experiments (Bernoulli process), each asking a ``yes-no'' question, and each with a two-valued outcome: success (with probability $p$) or failure (with probability $q = 1 - p$). Thus the Boolean algebra of events is $\mathcal{P}(\Omega)$, where $\Omega = \{1,\dots, n\}$ and atoms consists of all the sequences (regardless of the order) of successes and failures. The probability $x_\alpha = m(a_\alpha)$ of an atom $a_{\alpha}$ of $\mathcal{P}(\Omega)$ is:
$$x_{\alpha} = {n \choose \alpha  } p^{\alpha}(q)^{n-\alpha},$$
with $p\in (0,1)$ and $q = 1-p$. For the sake of simplicity, we set $p = q = \frac{1}{2}$. Relying on Remark \ref{rem:omogeneo}, it is easy to check that $M(x_{0},\dots, x_{3}), M(x_{0},\dots, x_{4}) > 0$. On the other hand, for $n=5$, we have $(1 + \frac{1}{5} + \frac{1}{10} + \frac{1}{10} + \frac{1}{5} + 1)^{2} - 4(1 + \frac{1}{25} + \frac{1}{100} + \frac{1}{100} + \frac{1}{25} + 1) = -\frac{41}{25}$. 
\end{example}

\section{Main results}\label{sez: principale}

Let $\mathcal{M}(\A)$ be the space of probability measures over a finite Boolean algebra $\A$, with $|A|=2^{k+1}$ and $\mathcal {M(\mathrm{At}(\A))}$ the space of probability measures over the atoms of $\A$ ($|\mathrm{At}(\A)| = k+1$). $\mathcal {M(\mathrm{At}(\A))}\subset (0, 1)^{k+1}$ is a convex (open)
subspace of $(0, 1)^{k+1}$ (we are considering only the positive measures). $\mathcal {M(\mathrm{At}(\A))}$ can be naturally identified  
with  the space $(0, 1)^{k+1}=(0, 1)\times\cdots\times (0, 1)$ by identifying  a measure
$m\in \mathcal {M(\mathrm{At}(\A))}$ with its values $\vx = (x_0, x_1, \dots , x_k)\in \mathbb{R}_+^{k+1}$, $x_\alpha=m(a_\alpha)$, $\alpha=0, \dots, k$.

The space $\mathcal {M(\mathrm{At}(\A))}$ can be described as: 
%Let $\mathcal P{M(\mathrm{At}(\A))}\subset M(\mathrm{At}(\A))$ be the set of probability measure, namely 
%$$\mathcal P{M(\mathrm{At}(\A))}=\mathcal {M(\mathrm{At}(\A))}\cap \Pi_{k},$$
$$\mathcal{M}(\mathrm{At}(\A)) = (0, 1)^{k+1}\cap \Pi_{k},$$
where $\Pi_{k}$ is the standard  $k$-simplex (or probability simplex) of $\mathbb{R}^{k+1}$, namely
$$\Pi_{k}=\{\vx\in (0, 1)^{k+1}\ | \ \sum_{\alpha=0}^kx_{\alpha}=1\}.$$

At the light of the above discussion, we define the space of measures 
${\mathcal M}_{ind}(\mathrm{At}(\A))$ induced by the flat metric of $\mathbb{R}^{N}$, namely those measures $m$ such that  $(\mathrm{At}(\A), d_m)$ admits an isometric embedding into some Euclidean space $\mathbb{R}^N$. By Morgan's theorem, $m\in {\mathcal M}_{ind}(\mathrm{At}(\A))$ if and only if the metric space  $(\mathrm{At}(\A), d_m)$ is flat, hence:

%\textcolor{blue}{Non mi piace, perch\'e ${\mathcal M}_{ind}(\mathrm{At}(\A))$ e ${\mathcal M}_{{flat}}(\mathrm{At}(\A))$ hanno senso solo se $\mathcal{M}$ \`e lo spazio delle misure di probabilit\`a (altrimenti non si ha una metrica).}

%In the same vein one defines the flat probability measures  ${\mathcal PM}_{flat}(\mathrm{At}(\A))$ and the induced probability measures as 
%$${\mathcal PM}_{flat}(\mathrm{At}(\A))={\mathcal M}_{flat}(\mathrm{At}(\A))\cap \Pi_{k}$$
%$${\mathcal PM}_{ind}(\mathrm{At}(\A))={\mathcal M}_{ind}(\mathrm{At}(\A))\cap \Pi_{k}.$$

%Under the previous identification, by Morgan's theorem and by the fact that $\mathrm{At}(\A)$ is finite we can write
%$${\mathcal M}_{ind}(\mathrm{At}(\A))={\mathcal M}_{flat}(\mathrm{At}(\A))=\bigcap_{n=3}^k
%C_n$$
\begin{equation}\label{ugvarie}
{\mathcal M}_{ind}(\mathrm{At}(\A))=\bigcap_{n=3}^k
C_n\cap\Pi_{k},
\end{equation}

where 
\begin{equation}\label{defcn}
C_n=\{\vx\in \mathbb{R}_+^{k+1}\ |\ \det (\M)\geq 0 \}, \text{ with }\  3\leq n\leq k.
\end{equation}

Notice that we are taking $\vx\in \mathbb{R}_+^{k+1}$ and not $\vx\in (0, 1)^{k+1}$. % and that the first equality in \eqref{ugvarie} holds as, talking about induced measures, we are not fixing a specific $N$ for the embedding (see also Remark \ref{rem: dimensione N}). 
We are interested in the solution of the following.

%\textcolor{red}{l'uguaglianza vale in quanto non stiamo fissando $N$}
\vspace{5pt}

\noindent
{\bf Problem.}
Study the topology of $\mathcal{M}_{ind}(\mathrm{At}(\A))$, and of its complement $\mathcal{M}(\mathrm{At}(\A))\setminus \mathcal{M}_{ind}(\mathrm{At}(\A))$, with the topology induced by $(0, 1)^{k+1}\subset\mathbb{R}_+^{k+1}$.
\vspace{5pt}

%We are not able  to tackle this problem in the general case.

The main contribution of the present paper is the solution to the above mentioned problem (see Theorem \ref{th: principale}). In order to tackle it, we begin by analyzing the topology of $C_n$.
%We begin by analyzing the topology of each $\widetilde C_n=C_n\cap\Pi_k$ and of its complement $(0,1)^{k+1}\setminus \widetilde C_n$.

\begin{lemma}\label{lemma: Cn contraibile}
For each $3\leq n \leq k$, the  space $C_n\cong H_n\times \mathbb{R}_+^{k-n}$ where $H_n$ is  a solid  half-hypercone  
in $\mathbb{R}_+^{n+1}$.
\end{lemma}
\begin{proof}
Consider the involutive homeomorphism 
\begin{equation}
\Phi: \mathbb{R}_+^{k+1}\rightarrow \mathbb{R}_+^{k+1}, \vx=(x_0, \dots ,x_k)\mapsto ( \frac{1}{x_0},\dots, \frac{1}{x_{k}}).
\end{equation}

%\begin{equation}\label{omeo}
%\Phi: \mathbb{R}_+^{k+1}\rightarrow \mathbb{R}_+^{k+1}, \vx=(x_0, \dots ,x_k)\mapsto ( x_1\cdots x_k,\dots, x_0\cdots\hat x_\alpha\cdots x_k, \dots, x_0\cdots x_{k-1}),
%\end{equation}
%whose $\alpha$-component the product $x_0\cdots\hat x_\alpha\cdots x_k$ (where $x_{\alpha}$ is omitted)
%with inverse 
%\begin{equation}\label{omeoparzinv}
%\Phi^{-1}\colon \mathbb{R}_+^{k+1}\rightarrow \mathbb{R}_+^{k+1}, (z_0, \dots ,z_k)\mapsto (z_0\cdots z_k )^{\frac{1}{k}} (z_0^{-1}, \dots , z_k^{-1}).
%\end{equation}

\noindent
In view of Lemma \ref{lemma: determinante} and Remarks \ref{m2} and \ref{rem:omogeneo} the image of $C_n$ via $\Phi$ is given by 

$$\Phi(C_n)=\{(z_0, \dots ,z_k)\in \mathbb{R}_+^{k+1} \ |\ \left(z_0+\cdots +z_n\right)^2- (n-1)(z_0^2+\cdots +z_n^2)\geq 0, \; 3\leq n\leq k\},$$
%where $z_{\alpha} = x_0\cdot\cdots \cdot\hat x_\alpha\cdot \cdots \cdot x_n$,
Equivalently,
$$\Phi(C_n)=\{(z_0, \dots ,z_k)\in \mathbb{R}_+^{k+1} \ |\ (n-2)\sum_{\alpha=0}^nz^2_\alpha -\sum_{\alpha, \beta=0}^n z_\alpha z_\beta\leq 0 \}  $$

%$$ = \{ (z_0, \dots ,z_k)\in \mathbb{R}^{k+1} \ |\ (n-2)\sum_{\alpha=0}^nz^2_\alpha-\sum_{\alpha, \beta=0}^n z_\alpha z_\beta\leq 0 \}\cap\mathbb{R}^{k+1}_{+}.$$
%Hence, by definition of $\Phi$,

%$$ \Phi_{n}(C_n) = (\Phi_{n}(C_n)\times\mathbb{R}^{n-k})\cap\mathbb{R}^{k+1}_{+} = (\Phi_{n}(C_n)\cap\mathbb{R}_{+}^{n+1})\times\mathbb{R}^{n-k}_{+} = \varphi_{n}(C_n) \times\mathbb{R}^{n-k}_{+}.$$
%\noindent
\noindent
We claim that, for every $3\leq n\leq k$, $\Phi(C_{n})$ is  affinely homeomorphic the product of  a solid  half-hypercone  
$H_n\subset\mathbb{R}_+^{n+1}$ with $\mathbb{R}_+^{k-n},$ i.e. $\Phi(C_{n})\cong H_n\times\mathbb{R}_+^{k-n}\subset\mathbb{R}_+^{k+1}$. In order to show the claim, observe that 
$$\Phi(C_{n}) =\{(z_0, \dots ,z_k)\in \mathbb{R}_+^{k+1} \ |\; \vec{z}^{\;t}A\vec{z}\leq 0 \},$$ 
where $A$ is the matrix (of order $n+1$)
$$A =  \begin{pmatrix}
n-2 & -1 & \cdots & -1 \\
-1 & n-2 & \cdots & -1 \\
\vdots & \vdots & \ddots & \vdots \\
-1 & -1 & \cdots & n-2 \\
\end{pmatrix}. $$

%and observe that 
%$$\Phi_n(C_n)=\{(z_0, \dots ,z_n)\in \mathbb{R}_+^{n+1} \ |\; \vec{z}^{\;t}A\vec{z}\leq 0 \}.  $$
\noindent
It is not difficult to check that the eigenvalues of $A$ are $\lambda_{0}=\lambda_{1}=\dots=\lambda_{n-1} = n-1$ and $\lambda_{n} = -2$ and an orthonormal basis of eigenvectors 
$$v_{0} = \frac{1}{\sqrt2}\begin{pmatrix}
1  \\
-1 \\
0 \\
\vdots\\
0  \\
\end{pmatrix} ,\  v_{i} =c_{i}\begin{pmatrix}
1  \\
\vdots \\
1 \\
-(i+1)\\
0  \\
\vdots
\end{pmatrix},\  v_{n} =\frac{1}{\sqrt{n+1}}\begin{pmatrix}
1  \\
1 \\
1 \\
\vdots\\
1  \\
\end{pmatrix},$$ 
with $c_{i}=\frac{1}{\sqrt{(i+1)(i+2)}}$, for $i=\{1,\dots,n-1\}$, and $1 $ in the first $i+1$ entries. 
Thus if $P=(v_0\dots v_n)$ is the associated  orthogonal matrix and $D$ the diagonal matrix of eigenvalues we can write

$$\Phi(C_{n}) =\{(z_0, \dots ,z_k)\in \mathbb{R}_+^{k+1} \ |\; \vec{z}^{\;t}PDP^{t}\vec{z}\leq 0 \}=
 \{(y_0, \dots ,y_k)\in \mathbb{R}^{k+1} \ |\; \vec{y}^{\;t}D\vec{y}\leq 0 \}\cap\{P\vec{y} > 0\}$$

$$ = \{(y_0, \dots ,y_k)\in \mathbb{R}^{k+1} \ |(n-1)^{2}(y_{0}^{2} + \dots + y_{n-1}^{2}) - 2y_{n}^{2}\leq 0 \}\cap\{P\vec{y} > 0\}, $$
\noindent
with $\vec{y} = P^{t}\vec{z}$. The affine transformation $P^{-1}=P^t\colon \mathbb{R}^{k+1}\rightarrow \mathbb{R}^{k+1}$ shows that, for every $3\leq n\leq k$, 
$\Phi(C_{n})$ is affinely homeomorphic to
$$(H_n\cap \mathbb{R}_+^{n+1})\times\mathbb{R}_+^{k-n},$$
where $H_n=\{(y_0, \dots ,y_n)\in \mathbb{R}^{n+1} \ |(n-1)^{2}(y_{0}^{2} + \dots + y_{n-1}^{2}) - 2y_{n}^{2}\leq 0 \}$
is a solid hypercone in $\mathbb{R}^{n+1}$.
In order to prove the claim we have to show that $H_n\subset \mathbb{R}_+^{n+1}$.

Observe that $H_n$ is indeed  obtained by a  rotation around the line generated by $(0, \dots, 0,  1)$ of angle $\alpha$, with $\cos(\alpha) = \sqrt{\frac{n-2}{n}}$.
Hence $P(H_n)$ is a solid hypercone obtained by a  rotation around the line generated by $(1, \dots, 1,  1)\in\mathbb{R}^{n+1}$ of the same angle. On the other hand, the angle $\alpha_{i}$ between the vector $(1,\dots,1)$ and each vectors of the canonical basis  $e_{i}$ of $\mathbb{R}^{n+1}$ satisfies $\cos(\alpha_{i}) = \frac{1}{\sqrt{n+1}} $, for every $i\in\{0,\dots, n\}$. 
 Thus, since $n\geq 3$, 
 $$\sqrt{\frac{n-2}{n}}=\cos(\alpha) > \cos (\alpha_{i})=\frac{1}{\sqrt{n+1}}, \ i=0, \dots , n.$$
It follows that $\alpha < \alpha_i$ therefore $P(H_n)\subset\mathbb{R}_{+}^{n+1}$ proving our claim and concluding the proof of the lemma.
\end{proof}

%\begin{lemma}\label{lemma: topologia del complementare}
%The complement $(0,1)^{k+1}\setminus \widetilde C_n$ is simply-connected (but not contractible), for every $3\leq n \leq k$.
%%\textcolor{blue}{Il complementare di $\widetilde C_n$ \`e semplicemente connesso per $n > 3$, connesso e con gruppo fondamentale $\mathbb{Z}$ per $n=3$.}
%\end{lemma}

%Even when  $k =4$ to understand the topology of ${\mathcal PM}_{ind}(\mathrm{At}(\A))$ seems to be an hard problem (the intersection between contractible sets can be very wild).
%On the other hand when $k=3$ we have the following:

%In the special case $k=3$ (a metric Boolean algebra with 4 atoms) the previous lemmas yield the following.
%
%\begin{theorem}
%Let $\A$ be a finite Boolean algebra such that $|\mathrm{At}(\A))|=4$.  Then the space 
%of all the (finitely additive) probability measures $\mathcal{M}(\mathrm{At}(\A)) $  is  the (disjoint) union a contractible set (the induced measures) and a  simply-connected (not contractible) space.
%\end{theorem}
%\begin{proof}
%\textcolor{red}{FOLLOWS BY THE PROOF OF THE LEMMA IN QUANTO LO SPAZIO IN QUESTO CASO UGUAGLIA $\widetilde C_3$  MA VA SCRITTA FACENDO VEDERE CHE $m$ IS AN OPEN MAP}
%\end{proof}

In order to provide an answer to the above stated problem, we need to prove some rather technical lemmas, whose proofs consist of adaptations of the proof techniques used in a standard results for compact convex subsets (with non-empty interior) in $\mathbb{R}^{k}$ (see e.g. \cite[Proposition 5.1]{Lee2010}).

Recall that a set $X\subseteq S^{k}$ in the unit $k$-dimensional sphere $S^{k}\subset \mathbb{R}^{k+1}$ is geodesically convex if, for every $x_1 , x_2 \in X$ there exits a (unique) minimal geodesic connecting them in $X$. 

\begin{lemma}\label{lemma 1: intersezione in 1 punto}
Let $X\subseteq S_{+}^{k} = S^{k}\cap\mathbb{R}_{+}^{k+1}$ be a compact, geodesically convex set with $\mathrm{Int}(X)\neq\emptyset$. Then any geodesic starting in $p\in\mathrm{Int}(X)$ intersects $X$ in exactly one point.
\end{lemma}
\begin{proof}
Let $g$ be a geodesic starting in $p\in \mathrm{Int}(X)$. Since $X$ is compact, then $X\cap g$ is compact and thus bounded. So let $x_{0}$ be the point having the maximal (spherical) distance from $p$, $x_{0} = \max\{d(x,p)\;|\; x\in X\}$. Let $B_{r}(p)$ an open (spherical) ball contained in $X$ and $I_{c} = \{\overline{yx_{0}} \;|\; y\in B_{r}(p) \text{ and } \overline{yx_{0}} \text{ the geodesic segment connecting } y \text{ to } x_{0} \}$ ($I_{c}$ is the ``ice-cream cone'' formed by the geodesic from $B_{r}(p)$ to $x_{0}$). Since $X$ is geodesically convex, then, for every $x\in g\smallsetminus\{x_{0}\}$, there is an open ball $B_{r_x}(x)\subseteq X$. This holds for every $x\in g\smallsetminus\{x_{0}\}$ and thus shows that $X\cap g = \{x_0\}$.
\end{proof}

\begin{lemma}\label{lemma 2: omeomorfismo disco}
Let $K\subset\mathbb{R}^{k+1}$ be a compact, star-shaped space with respect to $0$, with $\mathrm{Int}(K)\neq\emptyset$ and such that any ray from $0$ intersects $K$ in exactly one point. Then there exists a homeomorphism $F\colon \overline{B_1(0)}\to K$ such that $F(S^{k})= \partial K$, where $B_1(0)\subset \mathbb{R}^{k+1}$ denotes the open unit ball centered at the origin.
\end{lemma}
\begin{proof}
Define $f\colon\partial K\to S^{k}$, $x\mapsto f(x)=\frac{x}{|x|}$, i.e. $f(x)$ is the intersection of the ray (from $O$) with the sphere $S^{k}$. By definition, $f$ is continuous and, since any ray from $0$ intersects $S$ in exactly one point, is also bijective. Since $\partial K$ is compact, $f$ is a homeomorphism by the closed map lemma. We then define $F\colon \overline{B_1(0)}\to K$ as $x\mapsto\begin{cases} |x|f^{-1}(\frac{x}{|x|}) & \text{ if } x\neq 0, \\
0, & \text{ if } x= 0.
\end{cases}$
$F$ takes every radial segment $\overline{0x}$ with $x\in S^{k-1}$ in the radial segment $\overline{0f^{-1}(x)}$ in $K$, with $f^{-1}(x)\in\partial K$  (it is well defined since 
$S$ is star-shaped in $0$).
$F$ is continuous as $f^{-1}$ is and $\displaystyle\lim_{x\to 0}F(x) = 0$; it is injective as any ray from $0$ intersects $K$ in exactly one point and surjective as any $y\in K$ belongs to some ray. Finally, since $K$ is compact, $F$ is a homeomorphism by the closed map lemma. 
\end{proof}

\begin{lemma}\label{lemma 3: corona}
Let $K_1, K_2\subseteq\mathbb{R}^{k+1}$ be compact, star-shaped spaces with respect to $0$, with $\mathrm{Int}(K_{i})\neq\emptyset$ (for $i=1,2$), $K_1\subset \mathrm{Int}(K_2)$ and such that any ray from $0$ intersects $\partial K_{i}$ (for $i=1,2$) in exactly one point. Then there exists a homeomorphism $F\colon \overline {B_2(0)}\setminus B_1(0)\to K_2\setminus\mathrm{Int}(K_{1})$ such that $F(S_{1}^{k-1})= \partial K_1 $ and $F(S_{2}^{k-1})= \partial K_2 $. In particular
$$\Int K_2\setminus K_1\cong B_2(0)\setminus \overline{B_1(0)}\cong S^k\times (0, 1).$$
\end{lemma}
\begin{proof}
By Lemma \ref{lemma 2: omeomorfismo disco}, there are two homeomorphisms $F_1\colon S_{1}^{k-1}\to \partial K_1 $, $F_2\colon S_{2}^{k-1}\to \partial K_2 $. Define 
$$F\colon \overline {B_2(0)}\setminus B_1(0))\to K_2\setminus\mathrm{Int}(K_{1}), \ x\mapsto F(x) = (2-|x|)F_{1}(\frac{x}{|x|}) + (|x|-1)F_{2}(\frac{2x}{|x|}).$$ 
By definition (and the continuity of $F_1$, $F_2$), $F$ is continuous and surjective. To see that $F$ is injective, observe that, for points $x_1, x_2$ belonging to different rays, this follows from the assumption that every ray from $0$ intersects $\partial K_{i}$ ($i=1,2$) in exactly one point. Differently, let $x_2 = \lambda x_{1}$ (w.l.o.g. $\lambda > 0$) and $F(x_{1}) = F(x_{2})$. Upon setting $F_{1}(\frac{x}{|x|}) = a\in\partial K_1$ and $F_{2}(\frac{x}{|x|}) = b\in\partial K_1$, we have $(2-|x_1|)a + (|x_1|-1)b = (2-|x_2|)a + (|x_2|-1)b$, thus $(|x_2| - |x_1|)(a-b) = 0$ and, since $a\neq b$, $|x_1| = |x_2|$, i.e. $x_1 = x_2$ (as $x_2 = \lambda x_{1}$). Finally, $F$ is a homeomorphism by the closed map lemma. The last part follows by restricting $F$ to $\Int K_2\setminus K_1$ and by the fact that the annulus $B_2(0)\setminus \overline{B_1(0)}$ is homeomorphic to $S^k\times (0, 1)$.
\end{proof}

The solution to the above presented problem is given by the following. 

\begin{theorem}\label{th: principale}
Let $k\geq 3$. Then:
\begin{enumerate}
\item
${\mathcal M}_{ind}(\mathrm{At}(\A))$  is  contractible.
\item
$\mathcal{M}(\mathrm{At}(\A))\setminus{\mathcal M}_{ind}(\mathrm{At}(\A))$  is  simply-connected (not contractible).

\end{enumerate}
\end{theorem}
\begin{proof}
(1) Consider the  (open) retraction \begin{equation}\label{eq: retrazione su Cn tilde}
s:\mathbb{R}_+^{k+1}\rightarrow \Pi_k\subset (0, 1)^{k+1} , \ \vx=(x_0, \dots , x_k)\mapsto \frac{\vx}{s(\vx)},
\end{equation}
where $s(\vx):=\displaystyle\sum_{\alpha =0}^kx_\alpha$.
Then $${\mathcal M}_{ind}(\mathrm{At}(\A))=\displaystyle \displaystyle\bigcap_{n=3}^kC_n \cap \Pi_k=s(\displaystyle\bigcap_{n=3}^kC_n)$$
is contractible being a strong deformation retract of $\displaystyle\bigcap_{n=3}^k C_n$ which is contractible  (it is homeomorphic to   $\displaystyle\bigcap_{n=3}^k H_n\times \mathbb{R}_+^{k-n}$ by Lemma 
\ref{lemma: Cn contraibile}).

(2) Let $C =\displaystyle \bigcap_{n=3}^{k} C_n$. Then $\mathcal{M}(\mathrm{At}(\A))\setminus{\mathcal M}_{ind}(\mathrm{At}(\A))$ is a strong deformation retract of $\mathbb{R}_{+}^{k+1}\setminus C$ using  \eqref{eq: retrazione su Cn tilde}.
Let  $X= C\cap \overline{S_{+}^{k}}$. Then  $X$ is compact (the intersection of closed sets in the compact $\overline{S_{+}^{k}}$), $C\cong X\times (0,+\infty)$ 
(by Remark \ref{rem:omogeneo}) and $\mathbb{R}_{+}^{k+1}\setminus C = (S_{+}^{k}\setminus X)\times (0,+\infty)$. Moreover, $X$ has non-empty interior (it follows from Corollary \ref{cor: immersione atomi con stessa probabilita} and the proof of Lemma \ref{lemma: Cn contraibile} that, for instance, $p=(1,\dots, 1)\in Int(X)$) and it  is geodesically convex as a subset of $S_{+}^{k}$. To see this, consider $x_{1},x_{2}\in X\subset C$. Let $\overline{{x_{1}x_{2}}}\in C$ be the segment connecting $x_{1}$, $x_{2}$ ($C$ is convex) and $H\subset C$ the hypercone in $\mathbb{R}^{k+1}$ generated by $\overline{{x_{1}x_{2}}}$ (the inclusion 
$H\subset C$ follows by Remark \ref{rem:omogeneo}).
Then $H\cap S_{+}^{k}\subset X$ is the minimal geodesic segment in $S^k$, since it is the intersection of $S_{+}^{k}$ with an  hyperplane containing  $0$, $x_1$ and $x_2$. 
Thus, by Lemma \ref{lemma 1: intersezione in 1 punto}, the geodesic from $p=(1,\dots, 1)$ intersects $X$ exactly in one point. Consider the stereographic projection $\pi$ from the $S^{k}\setminus \{-p\}$ to the tangent space $T_pS^k$ of the sphere $S^{k}$ at the point $p$, namely the homeomorphism which to a point $x$ of  
 $S^{k}\setminus \{-p\}$ associates the intersection of the line joining $x$ to $-p$ with $T_pS^k$.
 Thus  $K_1\coloneqq\pi(X)$ and $K_2\coloneqq \pi (\overline{S_{+}^{k}})$ are subsets in $\mathbb{R}^{k}\cong T_{p}S^k$ satisfying the assumptions of Lemma \ref{lemma 3: corona}.
 It follows that 
 $$\Int K_2\setminus K_1\cong S_{+}^{k}\setminus X\cong S^{k-1}\times (0,1)$$ 
 and therefore $\mathbb{R}_{+}^{k+1}\setminus C = S_{+}^{k}\setminus X\times (0,+\infty)\cong S^{k-1}\times (0,1)\times (0,+\infty)$ is homotopically equivalent to $S^{k-1}$, thus simply connected ($k\geq 3$) and not contractible.
\end{proof}

\begin{remark}\label{rem: dimensione N}
Consider  the set of measure which can be induced in $\R^N$ with a fixed $N$. 
Then one can show that this set is open in $\Pi_k$ for $N\geq k$ since each $C_n$
with $n\leq k$ is open.
\end{remark}

%\section{Embeddings in $\mathbb{R}^{2}$}

\section{Conclusion and future work}

The space of \emph{strictly positive} probability measures $\mathcal{M}(\A)$ of a finite Boolean algebra $\A$ is an open convex sets (the boundary being given by all probability measures) in $[0,1]^{n}$, with $n =\; \mid A\mid$. % is the cardina
The main contribution of the present work consists in ``splitting'' such convex space into the (disjoint) union of the contractible set ${\mathcal M}_{ind}(\mathrm{At}(\A))$, corresponding to those measures $m$ for which the space $(\mathrm{At}(\A), d_m )$ embeds isometrically in $\mathbb{R}^{N}$, and its complement. We have given examples of measures belonging to each of the two components: the measure corresponding to the principle of indifference (see Corollary \ref{cor: immersione atomi con stessa probabilita}) and the binomial distribution (see Example \ref{ex: misura che non si immerge}), respectively. As a remarkable consequence of our topological characterization, we get that $m\in \mathcal{M}_{ind}(\mathrm{At}(\A))$ if and only if $m$ is homeotopically equivalent to the measure of indifference; on the other hand, all measures $m$ for which $(\mathrm{At}(\A), d_m )$ can not be isometrically embedded in $\mathbb{R}^{N}$ are path connected (though not homeotopically equivalent) to the binomial distribution.

It is natural to wonder in which component of $\mathcal{M}(\mathrm{At}(\A))$ can be located well-known distributions finding applications in probability (e.g. the hypergeometrical distribution) and exploits the potential applications of the fact that measures belonging to the component $\mathcal{M}_{ind}(\mathrm{At}(\A))$ have the same homotopy. In order to tackle the former problem, refinements of the useful Lemma \ref{lemma: determinante} shall be found to ease calculations.

Our main result (Theorem \ref{th: principale}) relies on the topological characterization (see Lemma \ref{lemma: Cn contraibile}) of the objects $C_n=\{\vx\in \mathbb{R}_+^{k+1}\ |\ \det (\M)\geq 0 \} $ ($ 3\leq n\leq k$), for which the assumption about the finiteness of the Boolean algebra considered is crucial (the existence of atoms is only a byproduct). We leave to future work the extension of the present setting to infinite (atomic) Boolean algebras. 

Finally, we confined our attention to metric Boolean algebras. However, in the last decades the theory of probability has been extended to algebraic semantics of several non-classical logics, via the development of the so-called theory of states. We will dedicate future work to the study of the metric properties, for instance, of MV-algebras (a study initiated in \cite{RiecanMundici,metricMV}) equipped with a faithful state (see e.g. \cite{FlaminioKroupa,flaminioJSL,flaminio2019}) or involutive bisemilattices (see \cite{BONZIOstati}) and their isometric embeddability into Euclidean spaces.

\section{Appendix}\label{sec: Appendice}

This section is dedicated to the proof of Lemma \ref{lemma: determinante}, whose technicalities are not so important (we believe) for the reading of the whole message of the paper.

Recall that, given a (square) matrix $A$ of order $n$, the \emph{adjugate} $Adj(A)$ of $A$ is the transpose of the cofactor matrix of $A$. Equivalently, $Adj(A)$ is the matrix of order $n$ such that $A\cdot Adj(A) = Adj(A)\cdot A = \mathrm{det}(A)\cdot \mathbb{I}_{n}$, where $\mathbb{I}_{n}$ is the identity matrix of order $n$. We recall here a result from linear algebra (see e.g. \cite{harville2008matrix}) that we will use in the proof of Lemma \ref{lemma: determinante}.  
%In the sequel we will use the following result from linear algebra .

\begin{lemma}[Matrix determinant lemma]\label{lemma:Matrix determinant lemma}
Let $A$ be a matrix of order $n$ and $u,v$ column vectors in $\mathbb{R}^{n}$. Then 
\[
\det(A+uv^{t}) = \det(A) + v^{t}Adj(A)u.
\]
\end{lemma}

%\textcolor{red}{ATTENZIONE: DIMENSIONE DELLO SPAZIO EUCLIDEO $N$, IL GENERICO SIMPLESSO $n$, GLI ATOMI DI UN' ALGEBRA BOOLEANA $k+1$}

%\begin{proof}
\noindent
\textbf{Lemma 10.} \emph{Let $\A$ be a finite metric atomic Boolean algebra with $k+1$ atoms and $\M$, $2\leq n\leq k$ be the matrix defined above. Then % \textcolor{red}{METTEREI L'INDICE $\alpha =0, \dots n$ e $i, j=1, \dots n$.}
%\begin{equation}\label{fondlemma}
$$ \det (\M)=2^{n-1}\left[\left(\sum_{\alpha=0}^n x_0\cdot\cdots \cdot\hat x_\alpha\cdot \cdots \cdot x_n\right)^2-(n-1)\left(\sum_{\alpha=0}^n x_0^2\cdot\cdots\cdot\hat x_\alpha^2\cdots\cdot x_n^2\right)\right], 
$$
%\end{equation}
where $\hat x_i$ means that $x_i$ has to be omitted.}

\begin{proof}

%\textcolor{blue}{Controllare Notazione: $\M$.}

Preliminarily observe that $\M = A + vv^{t}$, where $A $ %=\{A_{ij}\}_{i,j=1,\dots, n}$
is the matrix whose generic entry is $A_{ij} = -2(1-\delta_{ij})x_{i}x_{j}$ and $v= (x_{0} + x_1, \dots,  x_0 + x_{n})\in\mathbb{R}^{n}$.

\vspace{5pt}
\noindent
\textbf{Claim 1.} $\det(A) = -2^{n}x_{1}^{2}\dots x_{n}^{2}(n-1)$. \\
\noindent
Observe that $\det(A) = (-2)^{n}x_{1}^{2}\dots x_{n}^{2}\det(B)$, where $B =  \begin{pmatrix}
0 & 1 & \cdots & 1 \\
1 & 0 & \cdots & 1 \\
\vdots & \vdots & \ddots & \vdots \\
1 & 1 & \cdots & 0 \\
\end{pmatrix} $, i.e. $B_{ij}=1-\delta_{ij}$. Moreover, $B = cc^{t}-\mathbb{I}_{n}$, for $c=(1,\dots,1)$. Thus, by Lemma \ref{lemma:Matrix determinant lemma}, $\det(B) = \det(-\mathbb{I}_{n}) + c\mathrm{Adj}(-\mathbb{I}_{n})c^{t}$. Observe that $\mathrm{Adj}(-\mathbb{I}_{n}) = (-1)^{n-1}\mathbb{I}_{n}$ (indeed $-\mathbb{I}_{n}(-1)^{n-1}\mathbb{I}_{n} = (-1)^{n} \mathbb{I}_{n} = \det(-\mathbb{I}_{n})\cdot\mathbb{I}_{n}) $, hence 
\begin{align*}
 \det(B) =&  \det(-\mathbb{I}_{n}) + c^{t}\mathrm{Adj}(-\mathbb{I}_{n})c &&& \text{(Lemma \ref{lemma:Matrix determinant lemma})}
\\  =& (-1)^{n} + c^{t}(-1)^{n-1}\mathbb{I}_{n}c 
\\ = & (-1)^{n} + (-1)^{n-1}c^{t}c
\\ = & (-1)^{n} + (-1)^{n-1}n
\\ = & (-1)^{n-1} (n-1).
\end{align*}
It then follows that $\det(A) = (-2)^{n}x_{1}^{2}\cdot\dots\cdot x_{n}^{2}\det(B) = (-2)^{n}x_{1}^{2}\cdot\dots\cdot x_{n}^{2}(-1)^{n-1} (n-1) = -2^{n}x_{1}^{2}\cdot\dots\cdot x_{n}^{2}(n-1)$, showing Claim 1.

\vspace{5pt}
\noindent
\textbf{Claim 2.} $\mathrm{Adj}(A) = D$, with generic entry  
\[
D_{ij} = 2^{n-1}x_{1}^{2}\dots x_{n}^{2}\left (\frac{1}{x_i x_j}-\frac{\delta_{ij}}{x_{i}^{2}} (n-1)\right).
\]
By definition of adjugate, $A\cdot D = D\cdot A = \det(A)\cdot\mathbb{I}_{n}$, equivalently $\displaystyle\sum_{k=1}^{n}A_{ik}D_{kj} = \det(A)\cdot\delta_{ij}$. 
%\begin{align*}
 $\displaystyle\sum_{k=1}^{n}A_{ik}D_{kj} = \displaystyle\sum_{k=1}^{n}-2(1-\delta_{ik})x_{i}x_{k}\cdot\left(2^{n-1}x_{1}^{2}\dots x_{n}^{2}\left(\frac{1}{x_i x_j}-\frac{\delta_{kj}}{x_{k}^{2}} (n-1)\right)\right)  $ \\
$ = -2^{n}x_{1}^{2}\dots x_{n}^{2}\displaystyle\sum_{k=1}^{n}(1-\delta_{ik})\left( \frac{x_i}{x_j}-\frac{x_i \delta_{kj}}{x_{k}} (n-1)\right) $ \\
$ =  -2^{n}x_{1}^{2}\dots x_{n}^{2} \left( \frac{x_i}{x_j}\displaystyle\sum_{k=1}^{n}(1-\delta_{ik}) - x_i (n-1)\displaystyle\sum_{k=1}^{n}\frac{(1- \delta_{ik})\delta_{kj}}{x_{k}}\right) $ \\
$ =  -2^{n}x_{1}^{2}\dots x_{n}^{2} \left( \frac{x_i}{x_j}(n-1) - \frac{x_i}{x_j}(n-1) + x_{i}(n-1) \displaystyle\sum_{k=1}^{n}\frac{\delta_{ik}\delta_{kj}}{x_{k}}\right) $ \\
$ =   -2^{n}x_{1}^{2}\dots x_{n}^{2}(n-1)x_i \displaystyle\sum_{k=1}^{n}\frac{\delta_{ik}\delta_{kj}}{x_{k}}  =  \det(A)\cdot\delta_{ij}$,
%\end{align*}
which shows Claim 2. \\
\noindent
To simplify notation, let us fix $E_{ij} = \frac{1}{x_i x_j}-\frac{\delta_{ij}}{x_{i}^{2}} (n-1)$ (a part of the generic entry of $D$). 
In order to conclude, observe that, by Lemma \ref{lemma:Matrix determinant lemma}, %\textcolor{red}{il claim e il testo vanno fuori}
%\begin{align*}
$\det(\M) = \det(A) + vDv^{t} $ 
\\ 
$ = -2^{n}x_{1}^{2}\dots x_{n}^{2}(n-1) + 2^{n-1}x_{1}^{2}\dots x_{n}^{2}vEv^{t} $  \text{(Claims 1, 2)} \\
 $=  -2^{n-1}x_{1}^{2}\dots x_{n}^{2}\left(-2 (n-1) + vEv^{t}\right) $
\\ 
$ = -2^{n-1}x_{1}^{2}\dots x_{n}^{2}\left(-2 (n-1) + \displaystyle\sum_{i,j=1}^{n}(x_0 + x_{i})E_{ij}(x_{0} + x_j )\right) $
\\ 
$ =  -2^{n-1}x_{1}^{2}\dots x_{n}^{2}\left(-2 (n-1) + \displaystyle\sum_{i,j=1}^{n}(x_0 + x_{i})\left(\frac{1}{x_i x_j}-\frac{\delta_{ij}}{x_{i}^{2}} (n-1)\right)(x_{0} + x_j )\right) $
\\ $ =  -2^{n-1}x_{1}^{2}\dots x_{n}^{2} \left(-2 (n-1) + \displaystyle\sum_{i,j=1}^{n}\left (\frac{x_0}{x_{i}} + 1\right)\left(\frac{x_0}{x_j} + 1\right) - \displaystyle\sum_{i,j=1}^{n}\left(\frac{x_0}{x_{i}} + 1\right)(x_{0}+x_j )\frac{\delta_{ij}(n-1)}{x_i}\right) $
\\ 
$ =  -2^{n-1}x_{1}^{2}\dots x_{n}^{2} \left(-2 (n-1) + \left[\displaystyle\sum_{i=1}^{n}\left (\frac{x_0}{x_{i}} + 1\right)\right ]^{2} -(n-1) \displaystyle\sum_{i=1}^{n}\left(\frac{x_0}{x_{i}} + 1\right)^{2}\right) $
\\ 
$ =  -2^{n-1}x_{1}^{2}\dots x_{n}^{2} \left(-2 (n-1) + \left[n-1 +1+\displaystyle\sum_{i=1}^{n}\frac{x_0}{x_{i}}\right ]^{2} -(n-1) \left(\displaystyle\sum_{i=1}^{n}\frac{x_{0}^{2}}{x_{i}^{2}} +2\sum_{i=1}^{n}\frac{x_{0}}{x_{i}} + n+ 1 - 1\right)\right) $
\\ $ =  -2^{n-1}x_{1}^{2}\dots x_{n}^{2}(-2 (n-1) +\left(1+\displaystyle\sum_{i=1}^{n}\frac{x_0}{x_{i}}\right)^{2} + (n-1)^{2} +2(n-1)\displaystyle\sum_{i=1}^{n}\frac{x_0}{x_{i}} +2(n-1) - $
\\ 
$ -(n-1) \displaystyle\sum_{i=1}^{n}\frac{x_{0}^{2}}{x_{i}^{2}} -2(n-1)\displaystyle\sum_{i=1}^{n}\frac{x_{0}}{x_{i}} - (n-1)^{2} - (n - 1)) $
\\ 
$ =  -2^{n-1}x_{1}^{2}\dots x_{n}^{2}\left[ \left( 1 + \displaystyle\sum_{i=1}^{n}\frac{x_0}{x_i}\right)^{2} - (n-1)\left( 1 + \displaystyle\sum_{i=1}^{n} \frac{x_{0}^{2}}{x_{i}^{2}}\right)\right] $
\\ 
$ =  -2^{n-1}\left[ \left( x_1\dots x_n + \displaystyle\sum_{i=1}^{n}\frac{x_{0}x_{1}\dots x_{n}}{x_i}\right)^{2} - (n-1)\left(x_1^{2}\dots x_n^{2} + \displaystyle\sum_{i=1}^{n}\frac{x_{0}^{2}x_{1}^{2}\dots x_{n}^{2}}{x_i^{2}} \right)\right] $
\\ 
$ =  2^{n-1}\left[\left(\displaystyle\sum_{\alpha=0}^n x_0\cdot\cdots \cdot\hat x_j\cdot \cdots \cdot x_n\right)^2-(n-1)\left(\displaystyle\sum_{\alpha=0}^n x_0^2\cdot\cdots\cdot\hat x_j^2\cdots\cdot x_n^2\right)\right].$
%\end{align*}
\end{proof}

\section*{Acknowledgments}
The authors wish to thank Monica Musio and the members of the Department of Mathematics and Computer Science of the University of Cagliari for the useful comments received during the presentation of the contents of the paper. They also acknowledge the support of the INdAM, GNSAGA - Gruppo Nazionale per le Strutture Algebriche, Geometriche e le loro Applicazioni. %We finally thank Roberto Giuntini, Francesco Paoli, Tommaso Flaminio and two anonymous reviewers for their fruitful comments on a previous version of the paper.

%\bibliographystyle{abbrv}
%\bibliography{Misure_IBSL}

\end{document}